\documentclass[a4paper,12pt]{article}
\usepackage{amsmath,fancyhdr,amssymb,amsthm, geometry, enumerate, graphicx ,amsfonts,hyperref,mathrsfs}
\usepackage{authblk}
\usepackage[utf8]{inputenc}
\usepackage{tikz}
\usetikzlibrary{matrix,shapes,arrows,positioning,chains}

\author{Radhika Vasisht $^1$ and Ruchi Das $^{1,\dag}$}
\title{Generalizations of expansiveness in Non-Autonomous Discrete Systems } 

\theoremstyle{definition}
\newtheorem{defn}{Definition}[section]
\providecommand{\keywords}[1]{\textbf{Keywords :} #1}
\providecommand{\msc}[1]{\textbf{MSC(2010)} #1}
\theoremstyle{plain}
\newtheorem{thm}{Theorem}[section]

\newtheorem{Cor}{Corollary}[section]
\newtheorem{exm}{Example}[section]
\newtheorem{rmk}{Remark}[section]

\begin{document}
\date{}
\maketitle

\begin{abstract}
In this paper, we define and study generalizations of expansiveness, namely n- expansiveness, $\aleph_0$-expansiveness, continuum-wise expansive and meagre expansiveness for non-autonomous discrete dynamical systems. We discuss various properties of such non-autonomous systems and give necessary examples. We prove results related to non-existence of $\aleph_0$-expansive and meagre expansive non-autonomous system on certain spaces. We also study relation between $\aleph_0$-expansive and meagre-expansive non-autonomous systems.

\\ 
\msc{Primary 54H20; Secondary 37B55, 54B20}

\end{abstract}

\renewcommand{\thefootnote}{\fnsymbol{footnote}}
\footnotetext{\hspace*{-5mm}
\renewcommand{\arraystretch}{1}
\begin{tabular}{@{}r@{}p{15cm}@{}}
$^\dag$& the corresponding author. Email address: rdasmsu@gmail.com (R. Das)\\
$^1$&Department of Mathematics, University of Delhi, Delhi-110007, India\\
\end{tabular}}

\section{Introduction}
The theory of dynamical systems which deals with the study of systems governed by a consistent set of laws over time such as difference and differential equations is one of the very significant and applicable branches of mathematics.  Beginning with the contributions of Poincar\'e and Lyapunov, the study of dynamical systems has seen significant developments in the recent years.  This theory has gained the interest of reserchers worldwide owing to its useful connections with many different areas of mathematics \cite{MR1963683, MR3862772, MR3516121, vasisht2019induced}.

A system which has no external input and unfolds according to a same consistent law is the autonomous discrete dynamical system. The study of non-autonomous discrete dynamical systems helps in classifying the behaviour of various natural phenomenons which cannot be estimated with precision by autonomous systems. Most of the natural phenomenons, whether it is the pattern of day and night or the seasons or climatic conditions which vary over time are subjected to time-dependent extrinsic forces and their modeling leads to the theory of non-autonomous discrete dynamical systems. The mathematical theory of non-autonomous systems is thus more involved than the theory of autonomous systems. In the recent past, the theory of such systems has developed into a highly active field related to, yet considerably distinct from that of classical autonomous dynamical systems \cite{MR3584171, vasisht2018stronger, shao2019some, MR3770303, MR3842139, MR2960260, vasisht2018furstenberg}.  This development was motivated by problems of applied mathematics, in particular, in population biology and physics \cite{MR2293842, MR3173561}.

Let $(X,d)$ be a metric space and $\phi_n:X\rightarrow X,n\in \mathbb{N}$ be a continuous map. Consider the following non-autonomous discrete dynamical system (N D S) $(X,\phi_{1,\infty})$, where
\begin{equation} x_{n+1}=\phi_n(x_n) , n\in \mathbb{N}.\nonumber \end{equation}
For convenience we denote $({\phi_n})_{n =1}^\infty $ by $\phi_{1,\infty}$. Naturally, a difference equation of the form $x_{n+1}=\phi_n(x_n)$ can be thought of as the discrete analogue of a non-autonomous differential equation $\frac{dx}{dt}=\phi(x,t)$. 
Non-autonomous discrete dynamical systems were introduced by authors in \cite{MR1402417}. In various mathematical problems including those in the field of applied mathematics, we usually work with a sequence of maps instead of a single map.
\\In a dynamical system, orbits may have two opposing behaviors: to stay separated or stay close to each other. The classical notion of unstable homeomorphism (now known as expansiveness), introduced by Utz\cite{MR0038022}, deals with the former behavior of trajectories. Expansiveness is a very important and useful dynamical property. Roughly speaking, in an expansive dynamical system, every orbit can be accompanied by only one orbit with some certain constant. Expansive dynamical systems involve a large class of chaotic systems and in the last few decades an extensive study has been carried out on this property and its variants in both autonomous and non-autonomous systems \cite{MR3864292, MR3426052, MR3219409, MR3174280, das2018various}. One among these variants is the concept of n-expansiveness\cite{MR2837062} which weakens the restriction on every orbit thus allowing at most n companion orbits with a certain constant. The notion of $\aleph_0$-expansiveness\cite{MR3694815} is defined which allows at most countable companion orbits. Both these variants are weaken the notion of classical expansiveness and hence a natural question arises whether n-expansive systems and $\aleph_0$-expansive systems share the properties of expansive systems or not. In\cite{MR2837062, MR3694815}, these questions have been addressed for the autonomous systems. In this paper we attempt to answer such questions for non-autonomous dynamical systems. Recently, another variant of expansiveness, namely meagre-expansiveness was introduced and studied by authors in\cite{MR3759572} which provides a possible link between the expansive systems and descriptive set theory. We define and study meagre-expansive non-autonomous dynamical systems in this paper. We also define continuum-wise expansiveness for non-autonomous systems which helps in studying the relation between $\aleph_0$-expansiveness and meagre-expansiveness.
\\In section 2, we give preliminaries required for the other sections. In section 3, we define and study various properties of n-expansive and $\aleph_0$-expansive non-autonomous systems. We give a metric independent definition for both n-expansiveness and $\aleph_0$-expansiveness. We further prove that every $\aleph_0$-expansive non-autonomous system is cw-expansive. We establish the non-existence of $\aleph_0$-expansive equicontinuous non-autonomous system of homeomorphisms on an uncountable Lindel{\"o}f metric space. In section 4, we define and study various properties of meagre-expansive non-autonomous dynamical systems. We establish the non-existence of a meagre-expansive non-autonomous system of equicontinuous homeomorphism on a Lindel{\"o}f metric space. Further, we prove that on a locally connected metric space without isolated points, all $\aleph_0$-expansive systems are meagre-expansive.
\section{Preliminaries}
We recall the following concepts and notations.

For a metric space $X$ with metric $d$ and $\phi_n: X\rightarrow X, n\in \mathbb{N}$, a sequence of continuous maps, we call $\phi_{1,\infty}= \{\phi_n\}_{n=1}^{\infty}$ to be a non-autonomous system on $X$. The function $\phi_n\circ \phi_{n-1}\circ \cdots \circ \phi_1 $ for all $n=1,2,\ldots$ is denoted by $\phi_1^n$.
\\ For any $i\leq j$, $\phi_j\circ \phi_{j-1}\circ \cdots \circ \phi_{i+1}\circ \phi_i $ is denoted by $\phi_i^j$ and for $i>j$, $\phi_i^j$ is defined to be the identity map. For any $k>0$, we consider a non-autonomous map $\it[k^{th}-iterate$ of $\phi_{1,\infty}] (\phi_{1,\infty})^k = \{h_n\}_{n=1}^{\infty}$ on X, where $h_n= \phi_{nk}\circ \phi_{(n-1)k+k-1}\circ \cdots \cdots \circ \phi_{(n-1)k+2}\circ \phi_{(n-1)k+1}$, for all $n>0$.
\\For the non-autonomous system $h_{1,\infty}$ = $\{h_n\}_{n=1}^{\infty}$, where each $h_n$ is a homeomorphism, its inverse is defined to be $(h_{1,\infty}^{-1})=\{h_n^{-1}\}_{n=1}^{\infty}$.

\begin{defn}
A set $S\subseteq  X$ is said to be nowhere dense if $Int(\overline{S})=\emptyset$, where Int denotes the interior of a set and $\overline{S}$ denotes the closure of $A$.
\end{defn}

\begin{defn}
Let $(X,d)$ be a metric space and $\phi_n:X\rightarrow X$ be a sequence of continuous maps for $n=1,2,\ldots$. The non-autonomous system $(X,\phi_{1,\infty})$ is said to be $\it{equicontinuous}$ at the point $x_0$ if for every $\epsilon>0$, there exists $\delta>0$ such that $d(\phi_1^n(x_0),\phi_1^n(y))<\epsilon$, for all $n>0$, whenever $d(y,x_0)<\delta$.
\\The non-autonomous system $(X,\phi_{1,\infty})$ is said to be $\it{equicontinuous}$ if it is equicontinuous at every point $x_0\in X$.
\end{defn}

\begin{defn} \cite{MR3174280}
A homeomorphism $\psi:X_1 \rightarrow X_2 $ is said to be a $\it{uniform}$ $\it{homeomorphism}$ if $\psi$ is uniformly continuous on $X_1$ and $\psi^{-1}$ is uniformly continuous on $X_2$. 
\end{defn}

\begin{defn}\cite{MR3174280}
Let $(X_1,d_1)$ and $(X_2,d_2)$ be two metric spaces with non-autonomous maps $\phi_{1,\infty}=\{\phi_n\}_{n=1}^{\infty}$ and $\psi_{1,\infty}=\{\psi_n\}_{n=1}^{\infty}$ respectively. If there is a homeomorphism $h:X_1\rightarrow X_2$ such that $h\circ \phi_n = \psi_n\circ h$, for all $n=1,2,\ldots$, then $\phi_{1,\infty}$ and $\psi_{1,\infty}$ are said to be $\it{conjugate}$ or $\it{h-conjugate}$. 
\\Also $\phi_{1,\infty}$ and $\psi_{1,\infty}$ are said to be $\it{uniformly}$ $\it{conjugate}$ or $\it{uniformly}$ $\it{h-conjugate}$, if $h:X\rightarrow Y$ is a uniform homeomorphism. 
\end{defn}

\begin{defn}\cite{MR3174280}
Let $(X,\phi_{1,\infty})$ be a non-autonomous system and $X'$ be a subset of $X$. Then $X'$ is said to be $\it{invariant}$ under $\phi_{1,\infty}$ if $\phi_n(X')\subseteq X'$, for all $n>0$ or $\phi_1^n(X')\subseteq X'$, for all $n>0$.
\end{defn}

\begin{defn}\cite{MR3779662}
Let $(X,\phi_{1,\infty})$ be a non-autonomous system. A point $p$ is said to be a $\it{fixed}$ $\it{point}$ for $(X,\phi_{1,\infty})$ if $\phi_1^n(p)=p$, for all $n>0$.
\\Fix$(X,\phi_{1,\infty})$ denotes the set of all fixed points of the non-autonomous system $(X,\phi_{1,\infty})$.
\end{defn}

\begin{defn}
Let $(X,\phi_{1,\infty})$ be a non-autonomous system. A point $p$ is said to be a $\it{periodic}$ $\it{point}$ for $(X,\phi_{1,\infty})$ if there exists $n>0$ such that $\phi_1^{nk}(p)=p$, for any $k> 0$ \cite{MR3779662}.
\\Per$(X,\phi_{1,\infty})$ denotes the set of all periodic points of the non-autonomous system $(X,\phi_{1,\infty})$.
\end{defn}

The concept of n-expansiveness for autonomous systems was introduced by C.Morales. 
\begin{defn}\cite{MR2837062}
Let $(Y,d)$ be a metric space and $g:Y\rightarrow Y$ be a continuous map. The system $(Y,g)$ or $g$ is said to be $\it{n-expansive}$ if there exists a constant $c>0$ called n-expansiveness constant such that for every $x\in Y$, the set $\{y\in Y:d(g^i(x),g^i(y))\leq c,$ for all $ i\in \mathbb{N}\}$ has at most n elements.
\end{defn}
Similarly, $g:Y\rightarrow Y$ is a n-expansive homeomorphism, if for every $x\in Y$, the set $\{y\in Y:d(g^i(x),g^i(y))\leq c,$ for all $ i\in \mathbb{Z}\}$ has at most n elements.

\begin{defn}\cite{MR3694815}
Let $(Y,d)$ be a metric space and $g:Y\rightarrow Y$ be a continuous map. The system $(Y,g)$ or $g$ is said to be $\it{\aleph_0-expansive}$ if there exists a constant $c>0$ called $\aleph_0$-expansiveness constant such that for every $x\in Y$, the set $\{y\in Y:d(g^i(x),g^i(y))\leq c,$ for all $ i\in \mathbb{N}\}$ has at most countable elements.
\end{defn}
Similarly, $g:X\rightarrow X$ is a $\aleph_0$-expansive homeomorphism, if for every $x\in Y$, the set $\{y\in Y:d(g^i(x),g^i(y))\leq c,$ for all $ i\in \mathbb{Z}\}$ has at most countable elements.

By a continuum we mean a compact, connected and non-degenerate metric space. A $\it{subcontinuum}$ is a non-empty subset of a metric space $Y$ which is a continuum with respect to the induced topology.
\begin{defn}
A subcontinuum is $\it{degenerated}$ if it has only one point.
\end{defn}

\begin{defn}\cite{MR1222517}
Let $(Y,d)$ be a metric space and $g:Y\rightarrow Y$, be continuous map. The system $(Y,g)$ or $g$ is said to be $\it{continuum wise(cw)}$-$\it{expansive}$ if there exists a $c>0$ such that every non-degenerate subcontinuum $C$ of $X$ satisfies $diam(g^n(C))>c$ for some $n\in \mathbb{N}$.
\end{defn}
Similarly, $g:Y\rightarrow Y$ is a $\it{continuum wise(cw)}$-expansive homeomorphism, if there exists a $c>0$ such that every non-degenerate subcontinuum $C$ of $X$ satisfies $diam(g^n(C))>c$ for some $n\in \mathbb{Z}$.
\\Recently, in \cite{MR3759572}, authors introduced another variant of expansiveness, namely, meagre-expansiveness.

\begin{defn}\cite{MR3759572}
Let $(Y,d)$ be a metric space and $g:Y\rightarrow Y$ be a of continuous map. The autonomous system $(Y,g)$ or $g$ is said to be $\it{meagre-expansive}$ if there exists a constant $c>0$ called meagre-expansiveness constant such that for every $x\in Y$, the set $\{y\in Y:d(g^i(x),g^i(y))\leq c,$ for all $i\in \mathbb{N}\}$ is nowhere dense.
\end{defn}
Similarly, $g:Y\rightarrow Y$ is a meagre-expansive homeomorphism, if for every $x\in Y$, the set $\{y\in Y:d(f^i(x),f^i(y))\leq c,$ for all $ i\in \mathbb{Z}\}$ is nowhere dense.

\begin{defn}\cite{MR3759572}
A Borel probability measure $\mu$ of a metric $(Y,d)$ is $\it{meagre-expansive}$ with respect to $(Y,g)$ if there exists $c>0$(called meagre-expansiveness constant) such that $\mu(IntS_{c}(x))$=0, for all $y\in Y$, where $S_c(y)=\{x\in Y: d(g^i(x),g^i(y))\leq c,i\in \mathbb{Z}$\}.
\end{defn}
 
\begin{defn}\cite{MR3219409}
Let $(X,d)$ be a compact metric space and $(X,\phi_{1,\infty})$ be a non-autonomous system. A finite open cover $\mathcal{B}$ is said to be a generator of $\phi_{1,\infty}$, if for every bisequence $\{B_n\}$ of members of $\mathcal{B}$, $\bigcap_{-\infty}^{\infty}(\phi_1^n)^{-1}(\overline{B_n})$ has at most one point. 
\end{defn} 
 
\begin{defn}
A point $x\in X$ is said to be a $\it{stable}$ $\it{point}$ for $\phi_{1,\infty}$ if for every $\epsilon>0$, there is a $\epsilon'>0$ such that $d(\phi_1^i(x),\phi_1^i(y))\leq \epsilon$ for all $i\in \mathbb{Z}$, whenever $d(x,y)\leq \epsilon'$. The set of stable points of $(X,\phi_{1,\infty})$ is denoted by $Stab(X,\phi_{1,\infty})$.
\end{defn}

\begin{defn}\cite{MR2197587}
Let $(Y,d)$ be a metric space and $g_n:Y\rightarrow Y,$ $n=1,2,\ldots$ be a sequence of continuous maps. Then
\begin{itemize}
\item[i.] the $\it{\omega}$-limit set of a point $y\in Y$ is given by $\omega(g_{1,\infty},y)=\{y'\in Y:\lim\limits_{k\to \infty} d(g_1^{n_k}(y),y')=0\}$, for some strictly increasing sequence $\{n_k\}$ of integers.
\item[ii.] the $\it{\alpha}$-limit set of a point $y\in Y$ is given by $\alpha(g_{1,\infty},y)=\{y'\in Y:\lim\limits_{k\to \infty} d(g_1^{n_k}(y),y')=0\}$, for some strictly decreasing sequence $\{n_k\}$ of integers.
\end{itemize}
\end{defn}

We say that $y\in Y$ has converging semi-orbits under $g_{1,\infty}$ if both $\omega(g_{1,\infty},y)$ and $\alpha(g_{1,\infty},y)$ are singletons. 
\\The set of points having converging semi-orbits in $Y$, is denoted by $A(Y,g_{1,\infty})$.

\section{n-expansiveness and $\aleph_0$-expansiveness for\\ Non-autonomous Systems}
In this section we define and study n-expansive and $\aleph_0$-expansive non-autonomous discrete dynamical systems.

\begin{defn}
The non-autonomous system $(X,\phi_{1,\infty})$ or $\phi_{1,\infty}$ is said to be $\it{n-expansive}$ if there exists a constant $c>0$ (called as an n-expansiveness constant) such that for every $x\in X$, the set $\{y\in X:d(\phi_1^i(x),\phi_1^i(y))\leq c,$ for all $ i\in \mathbb{N}\}$ has at most n elements.
\\
Similarly, a sequence of homeomorphisms $\phi_n:X\rightarrow X$ is said to be n-$expansive$, if for every $x\in X$, the set $\{y\in X:d(\phi_1^i(x),\phi_1^i(y))\leq c,$ for all $ i\in \mathbb{Z}\}$ has at most n elements.
\end{defn}
\begin{defn}
The non-autonomous system $(X,\phi_{1,\infty})$ or $\phi_{1,\infty}$ is said to be $\it{\aleph_{0}-expansive}$ if there exists a constant $c>0$ (called as an $\aleph_{0}$-expansiveness constant) such that for every $x\in X$, the set $\{y\in X:d(\phi_1^i(x),\phi_1^i(y))\leq c,$ for all $i\in \mathbb{N}\}$ has at most countable elements.
\\Similarly, a sequence of homeomorphisms $\phi_n:X\rightarrow X$ is said to be $\aleph_0$-expansive, if for every $x\in X$, the set $\{y\in X:d(\phi_1^i(x),\phi_1^i(y))\leq c,$ for all $ i\in \mathbb{Z}\}$ has at most countable elements.
\end{defn}
In \cite{MR3694815}, authors have generalized the concept of generators introduced by Keynes and Robertson\cite{MR0247031}. We use this generalized notion in non-autonomous setting to obtain the following theorems.

\begin{thm}
Let $(X,d)$ be a compact metric space and $\phi_n:X\rightarrow X$; $n=1,2,\ldots$ be  a sequence of homeomorphisms. The non-autonomous system $(X,\phi_{1,\infty})$ is $n-expansive$ if and only if there is a finite  open cover $\mathcal{B}$ of $X$ such that for every bisequence $\{B_m\}_{m=-\infty}^{\infty}$ of members of $\mathcal{B}$, $\bigcap_{m=-\infty}^{\infty}(\phi_1^m)^{-1}(\overline{B_m})$ has at most n elements.
\end{thm}

\begin{proof}
Let $(X,\phi_{1,\infty})$ be n-expansive with constant of n-expansiveness $c>0$. Let $\mathcal{B}$ be a finite cover of $X$ consisting of open balls of radius $c/2$. Let $\{B_m\}$ be a bisequence of members of $\mathcal{B}$. Fix some $x\in \bigcap_{m\in \mathbb{Z}}(\phi_1^m)^{-1}(\overline{B_m})$. Then $\phi_1^m(x)\in \overline{B_m}$, for all $m\in Z$.

Now, for any $y\in \bigcap_{m\in \mathbb{Z}}(\phi_1^m)^{-1}(\overline{B_m})$, $d(\phi_1^m(x),\phi_1^m(y))\leq c$, for all $m\in \mathbb{Z}$.
\\Thus, by n-expansiveness of $(X,\phi_{1,\infty})$, we have that the set $\{y\in X :y\in \bigcap_{m\in \mathbb{Z}}(\phi_1^i)^{-1}(B_i)\}$ has almost n elements.

Conversely, let $\mathcal{B}$ be a finite open cover of $X$ such that for every bi-sequence $\{B_m\}_{m=-\infty}^{\infty}$ of members of $\mathcal{B}$, $\bigcap_{m\in \mathbb{Z}}(\phi_1^m)^{-1}(\overline{B_m})$ has at most n-elements. Let $c$ be a  Lebesgue number for $\mathcal{B}$. Fix some $x\in X$ and let $S=\{y\in X:d(\phi_1^m(x),\phi_1^m(y))\leq c,$ for all $m\in \mathbb{Z}\}$. Since $c$ is a Lebesgue number for $\mathcal{B}$; for every $y\in S$ and every pair $\{\phi_1^i(x),\phi_1^{i}(y)\}$, there exists $B_i\in \mathcal{B}$ such that  $\{\phi_1^i(x),\phi_1^{i}(y)\}\subseteq B_i$, which implies $y\in (\phi_1^i)^{-1}(B_i)$, for all $i\in \mathbb{Z}$ and for all $y\in S$. Thus, $S$ has at most n elements and hence $(X,\phi_{1,\infty})$ is n-expansive.
\end{proof}

\begin{thm}
Let $(X,d)$ be a compact metric space and $\phi_n:X\rightarrow X$; $n=1,2,\ldots$ be  a sequence of homeomorphisms. The non-autonomous system $(X,\phi_{1,\infty})$ is $\aleph_0-expansive$ if and only if there is a finite  open cover $\mathcal{B}$ of $X$ such that for every bi-sequence $\{B_m\}_{m=-\infty}^{\infty}$ of members of $\mathcal{B}$, $\bigcap_{m=-\infty}^{\infty}(\phi_1^m)^{-1}(\overline{B_m})$ has at most countable elements.
\end{thm}

\begin{proof}
Let $(X,\phi_{1,\infty})$ be $\aleph_0$-expansive with constant of $\aleph_0$-expansiveness $c>0$. Let $\mathcal{B}$ be a finite cover of $X$ consisting of open balls of radius $c/2$. Let $\{B_m\}$ be a bisequence of members of $\mathcal{B}$. Fix some $x\in \bigcap_{m\in \mathbb{Z}}(\phi_1^m)^{-1}(\overline{B_m})$. Then $\phi_1^m(x)\in \overline{B_m}$, for all $m\in Z$.

Now, for any $y\in \bigcap_{m\in \mathbb{Z}}(\phi_1^m)^{-1}(\overline{B_m})$, we have $d(\phi_1^m(x),\phi_1^m(y))\leq c$, for all, $m\in \mathbb{Z}$.
\\Thus, by $\aleph_0$-expansiveness of $(X,\phi_{1,\infty})$, we have that the set $\{y\in X:y\in \cap_{m\in \mathbb{Z}}(\phi_1^i)^{-1}(B_i)\}$ has at most countable elements.

Conversely, let $\mathcal{B}$ be a finite open cover of $X$ such that for every bi-sequence $\{B_m\}_{m=-\infty}^{\infty}$ of members of $\mathcal{B}$, $\bigcap_{m\in \mathbb{Z}}(\phi_1^m)^{-1}(\overline{B_m})$ has at most countable elements. Let $c$ be a  Lebesgue number for $\mathcal{B}$. Fix some $x\in X$ and let $S=\{y\in X:d(\phi_1^m(x),\phi_1^m(y))\leq c,$ for all $m\in \mathbb{Z}\}$. Since $c$ is a Lebesgue number for $\mathcal{B}$; for every $y\in S$ and every pair $\{\phi_1^i(x),\phi_1^{i}(y)\}$, there exists $B_i\in \mathcal{B}$ such that  $\{\phi_1^i(x),\phi_1^{i}(y)\}\subseteq B_i$, which implies $y\in (\phi_1^i)^{-1}(B_i)$, for all $i\in \mathbb{Z}$ and for all $y\in S$. Thus, $S$ has at most countable elements and hence $(X,\phi_{1,\infty})$ is $\aleph_0$-expansive.
\end{proof}

\begin{thm}
Let $(X_1,d_1,\phi_{1,\infty})$ and $(X_2,d_2,\psi_{1,\infty})$ be two non-autonomous systems such that $\phi_{1,\infty}$ is uniformly conjugate to $\psi_{1,\infty}$. If $\phi_{1,\infty}$ is n-expansive, then so is $\psi_{1,\infty}$.
\end{thm}
\begin{proof}
Let $c>0$ be an n-expansiveness constant for $\phi_{1,\infty}$. Since $\phi_{1,\infty}$ is uniformly conjugate to $\psi_{1,\infty}$, therefore there exists a uniform homeomorphism $h:X_1\rightarrow X_2$ such that $h\circ \phi_n=\psi_n\circ h$, for all $n>0$. Thus, $\phi_1^n\circ h^{-1}=h^{-1}\circ \psi_1^n$, for all $n>0$.
As $h^{-1}$ is uniformly continuous, therefore for every $c>0$, there is $c' >0$ such that for $x,y\in X_2$, $d_2(x,y)\leq c' \implies d_1(h^{-1}(x),h^{-1}(y))\leq c$.
\\For a fix $x\in X_2$, the set S = $\{y\in X_2:d_2(\psi_1^n(x),\psi_1^n(y))\leq c', \forall n>0\} =\{y\in X_2:d_1(h^{-1}(\psi_1^n(x)),h^{-1}(\psi_1^n(y)))\leq c', \forall n>0\} =\{y\in X_2:d_1(\phi_1^n(h^{-1}(x)),\phi_1^n(h^{-1}(y)))\leq c, \forall n>0\}$. 
\\Since $\phi_{1,\infty}$ is n-expansive with n-expansiveness constant $c>0$, therefore $S$ has atmost n elements and hence $\psi_{1,\infty}$ is n-expansive with n-expansiveness constant $c'>0$.
\end{proof}

Based on similar arguments, one can prove the following result.

\begin{thm}
If the non-autonomous system $\phi_{1,\infty}$ on $(X_1,d_1)$ is uniformly conjugate to a non-autonomous system $\psi_{1,\infty}$ on $(X_2,d_2)$ and $\phi_{1,\infty}$ is $\aleph_0$-expansive, then $\psi_{1,\infty}$ is also $\aleph_0$-expansive.
\end{thm}

\begin{Cor}
If a non-autonomous system $\phi_{1,\infty}$ is n/$\aleph_0$-expansive on a compact metric space $X$, then so is the non-autonomous system $\psi_{1,\infty}$ = $\{h\circ \phi_n\circ h^{-1}\}_{n=1}^{\infty}$, where $h$ is a self homeomorphism of $X$. 
\end{Cor}

\begin{thm}
Let $(X,f_{1,\infty})$ be a non-autonomous system, where $\phi_n:X\rightarrow X, n=1,2,\ldots$ is a sequence of equicontinuous maps. For any positive integer $k$, $\phi_{1,\infty}$ is n-expansive (respectively $\aleph_0$-expansive) if and only if $(\phi_{1,\infty})^k$ is n-expansive (respectively $\aleph_0$-expansive).
\end{thm}

\begin{proof}
Let $c>0$ be an n-expansiveness constant for $\phi_{1,\infty}$. Since $\{\phi_n\}_{n=1}^\infty$ is an equicontinuous family of maps, for any $m\geq 0$ and $mk+1\leq j\leq (m+1)k$; $\phi_j$ is uniformly continuous on $X$ and therefore there exists $c_j>0$ such that \begin{align*}
d(x,y)\leq c_j\implies d(\phi_{mk+1}^j(x),\phi_{mk+1}^j(y))\leq c.
\end{align*} 
Since $\phi_n:X\rightarrow X,n=1,2,\ldots$ are equicontinuous maps, $c_j$ doesn't depend on m. Take $c'=min\{c_j:mk+1\leq j \leq (m+1)k\}$. So for any $m\geq 0$, $d(x,y)\leq c'\implies d(\phi_{mk+1}^j(x),\phi_{mk+1}^j(y))\leq c.$ Let $(\phi_{1,\infty})^k=\psi_{1,\infty}=\{\psi_n\}_{n=1}^{\infty}$, where $\psi_n=\phi_{(n-1)k+1}^{nk}$ and $\psi_1^n=\psi_n\circ \ldots \circ \psi_1$.
Note that $\phi_1^{nk}=\psi_1^n$. Thus for any $m\geq 0$ and $mk\leq j\leq (m+1)k$, $d(\psi_1^m(x),\psi_1^m(y))\leq c$ which implies $d(\phi_1^{mk}(x),\phi_1^{mk}(y))\leq c'$ and hence we get that $d(\phi_{mk+1}^j(\phi_1^{mk}(x)),\phi_{mk+1}^j(\phi_1^{mk}(y)))\leq c$ which implies $d(\phi_1^j(x),\phi_1^j(y))\leq c$. Since $c$ is an n-expansiveness constant for $\phi_{1,\infty}$, the set $\{y\in X:d(\phi_1^i(x),\phi_1^i(y))\leq c ;i>0\}$ has at most n elements. Therefore,$\{y\in X:d(\psi_1^i(x),\psi_1^i(y))\leq c' ;i> 0\}$ has at most n elements and hence $(\phi_{1,\infty})^k$ is n-expansive with constant of n-expansiveness $c'>0$.
\\Conversely, if $(\phi_{1,\infty})^k$ is n-expansive with constant of n-expansiveness $c>0$, then for any $x\in X$, the set $\{d(\psi_1^i(x),\psi_1^i(y))\leq c,i> 0\}$ has at most n elements, where $\psi_n=\phi_{(n-1)k+1}^{nk}$. Thus, the set $\{y\in X:d(\phi_1^{ik}(x),\phi_1^{ik}(y))\leq c,$ for all $i> 0\}$ = $\{d(\phi_1^j(x),\phi_1^j(y))\leq c,j>0\}$ has at most n elements. Hence, $(\phi_{1,\infty})$ is n-expansive with constant of n-expansiveness $c>0$. 
\end{proof}

\begin{thm}
The non-autonomous systems $(X,\phi_{1,\infty})$ on a compact metric space is n-expansive (respectively $\aleph_0$-expansive) if and only if $(\phi_{1,\infty})^{-1}$ is n-expansive (respectively $\aleph_0$-expansive), where $\phi_{1,\infty}$ is a family of self-homeomorphisms of $X$.
\end{thm}

\begin{proof}
Let $c>0$ be an n-expansive constant for $\phi_{1,\infty}$. For a fixed $x\in X$; the set \begin{align*}
\{y\in X: d(\phi_1^i(x),\phi_1^i(y))\leq c,i\in \mathbb{Z}\}
\end{align*} has at most n elements, i.e, the set
\begin{align*}
\{y\in X: d((\phi_1^i)^{-1}(x),(\phi_1^i)^{-1}(y))\leq c,i\in \mathbb{Z}\}
\end{align*} has at most n-elements. Thus, $(\phi_{1,\infty})^{-1}$ is n-expansive.
\\The converse can be proved similarly.
\end{proof}
Based on similar arguments, one can prove the result for $\aleph_0$-expansive non-autonomous systems.

\begin{Cor}
Let $\phi_{1,\infty}$ be a family of equicontinuous self-homeomorphisms of a compact metric space $(Y,d)$. Then $(\phi_{1,\infty})$ is n-expansive (respectively $\aleph_0$-expansive) if and only if $(\phi_{1,\infty})^{k}$ is n-expansive (respectively $\aleph_0$-expansive), for every $k\in \mathbb{Z}-\{0\}$.
\end{Cor}

\begin{thm}
If the non-autonomous system $(X,\phi_{1,\infty})$ is n-expansive (respectively $\aleph_0$-expansive) and $Y\subseteq X$ be an invariant subset of $X$, then the restriction of $(\phi_{1,\infty})$ on $Y$ defined by $\phi_{1,\infty}|_Y=\{\phi_n|_Y\}_{n=1}^{\infty}$ is also n-expansive (respectively $\aleph_0$-expansive). 
\end{thm}

\begin{proof}
Let $c>0$ be a $n-expansive$ constant for $(X,\phi_{1,\infty})$. Fix some $x\in Y\subseteq X$, then the set $S=\{y\in X:d(\phi_1^i(x),\phi_1^i(y))\leq c ,i>0\}$ has at most n elements. Since $Y$ is invariant under $\phi_{1,\infty}$, we have $\phi_1^i(Y)\subseteq Y$, for all $i> 0$ and thus, the set $S_1=\{y\in Y:d(\phi_1^i(x),\phi_1^i(y))\leq c ; i>0\} \subseteq  S$. Therefore, $\phi_{1,\infty}|_Y$ is also n-expansive with n-expansiveness constant $c>0$.
\end{proof}
Similarly, one can prove the result for $\aleph_0$-expansiveness.

Let $(X_1,d_1)$ and $(X_2,d_2)$ be two metric spaces and $\phi_{1,\infty}$ and $\psi_{1,\infty}$ be a sequence of continuous maps on $X_1$ and $X_2$ respectively . The metric $d$ on the product $X_1\times X_2$ is defined by 
\begin{align*}
d((x_1,x_2),(y_1,y_2))= max \{d_1(x_1,y_1),d_2(x_2,y_2)\}, (x_1,x_2), (y_1,y_2)\in  X_1\times X_2.
\end{align*}

\begin{thm}

The non-autonomous system $(X_1,\phi_{1,\infty})\times (X_2,\psi_{1,\infty})=(X_1\times X_2,\{\phi_n\times \psi_n\}_{n=1}^{\infty})$ is n-expansive (respectively $\aleph_0$-expansive) on $X_1\times X_2$, whenever $(X_1,\phi_{1,\infty})$ and $(X_2,\psi_{1,\infty})$ are both n-expansive (respectively $\aleph_0$-expansive). Hence every finite direct product of n-expansive (respectively $\aleph_0$-expansive) non-autonomous systems is n-expansive (respectively $\aleph_0$-expansive).

\end{thm}

\begin{proof}
Let $c_1$ and $c_2$ be n-expansiveness constants for $\phi_{1,\infty}$ and $\psi_{1,\infty}$ respectively and choose $c=min\{c_1, c_2\}$. For a fix $(x_1,x_2)\in X_1\times X_2$, suppose that the set $S=\{(x_1',x_2'):d((\phi_1^i\times \psi_1^i)(x_1,x_2)),(\phi_1^i\times \psi_1^i)(x_1',x_2')))\leq c;$ for all $i>0\}$ has more than n elements. Then the set $\{(x_1',x_2'):d((\phi_1^i(x_1),\psi_1^i(x_2)),(\phi_1^i(x_1'),\psi_1^i(x_2')))\leq c,$ for all $i>0\}$ has more than n elements which implies that the set $\{(x_1',x_2'):$max$\{d_1(\phi_1^i(x_1),\phi_1^i(x_2')),d_2(\psi_1^i(x_2),\psi_1^i(x_2'))\}\leq c,$ for all $i>0\}$ has more than n elements. Therefore, the sets $\{x_1':d(\phi_1^i(x_1),\phi_1^i(x_1'))\leq c\leq c_1\}$ and $\{x_2':d(\psi_1^i(x_2),\psi_1^i(x_2'))\leq c\leq c_2\}$ have more than n elements for all $i>0$, which contradicts the n-expansiveness of $\phi_{1,\infty}$ and $\psi_{1,\infty}$. Thus, the set $S$ can have at most n elements and hence $\phi_{1,\infty}\times \psi_{1,\infty}$ is n-$expansive$.
\\The proof follows similarly for $\aleph_0$-expansiveness.
\end{proof}

\begin{thm}
Let $(X,d)$ be a compact metric space and $\phi_n:X\rightarrow X$, $n=1,2,\ldots$ be continuous maps.We have the following:
\begin{itemize}
\item[i.] If $(X,\phi_{1,\infty})$ is n-expansive, then the set Fix$(X,\phi_{1,\infty})$ is finite and hence the set Per$(X,\phi_{1,\infty})$ is countable. 
\item[ii.] If $(X,\phi_{1,\infty})$ is $\aleph_0$-expansive, then the set Fix$(X,\phi_{1,\infty})$ is countable and hence the set Per$(X,\phi_{1,\infty})$ is countable. 
\item[iii.] If $(X,\phi_{1,\infty})$ is $\aleph_0$-expansive, then the set $A(X,\phi_{1,\infty})$ is countable.
\end{itemize} 
\end{thm}

\begin{proof} 
Let $(X,d)$ be a compact metric space and $\phi_n:X\rightarrow X$, $n=1,2,\ldots$ be continuous maps.
\begin{itemize}

\item[i.]Let $(X,\phi_{1,\infty})$ be n-expansive with constant of n-expansiveness $c>0$. Suppose Fix$(X,\phi_{1,\infty})$ is infinite, then $X$ being compact, Fix$(X,\phi_{1,\infty})$ must have a limit point, say $p\in X$. Let $B_d(p,c/2)$ denote the open ball centred at p of radius $c/2$. It is easy to note that the set $S=$ Fix$(X,\phi_{1,\infty})\cap B_d(p,c/2)$ is infinite. For some fixed $x\in S$, let $S(x)$ denote the set $S(x)=\{y\in S:d(\phi_1^i(x),\phi_1^i(y))=d(x,y)\leq c,$ for all $i> 0\}$. Since $S$ is infinite, we get that $S(x)$ is infinite which contradicts the n-expansiveness of $(X,\phi_{1,\infty})$. Thus, we have $S$ is finite and hence Fix$(X,\phi_{1,\infty})$ is finite.
\\As the set of periodic points, Per$(X,\phi_{1,\infty})$, is the union of of fixed points of $(\phi_{1,\infty})^k$, for all $k\geq 0$, we have Per$(X,\phi_{1,\infty})$ is countable.

\item[ii.] The proof follows similarly as part [i].
\item[iii.] Let $(X,\phi_{1,\infty})$ be $\aleph_0$-expansive with $\aleph_0$-expansiveness constant $c>0$. By[ii], we have that Fix$(X,\phi_{1,\infty})$ is countable. Enumerate Fix$(X,\phi_{1,\infty})$ as $x_1,x_2,\ldots$ and let if possible $A(X,\phi_{1,\infty})$ be uncountable. Denote by $A(i,j,k)$ the set $\{x\in A(X,\phi_{1,\infty}):d(\phi_1^n(x),x_i)\leq c/2$ and $d(\phi_1^n(x),x_j)\leq c/2,$ for all $n\geq k\}$. Clearly, $A(X,\phi_{1,\infty})\subseteq  \bigcup\limits_{i,j,k\in \mathbb{N}}A(i,j,k)$. As $A(X,\phi_{1,\infty})$ is uncountable, there exist $i_0, j_0,k_0$ such that $A(i_0,j_0,k_0)$ is uncountable. $X$ being compact, we have $A(i_0,j_0,k_0)\subseteq  \bigcup\limits_{m=1}^{t}B_{d(\epsilon_m)}(x_m)$, where $B_{d(\epsilon_m)}(x_m)$ is an open ball of radius $\epsilon_m$ centred at $x_m$, such that $B_{d(\epsilon_m)}(x_m)=\{y\in X:$ if $ d(x_m,y)\leq \epsilon_m, $ then $d(\phi_1^n(x_m),\phi_1^n(y))\leq c/2,$ for all $ n\leq k_o \}$. So, we get $1\leq m_0\leq t$ such that $B_{d(\epsilon_{m_0})}(x_{m_0})$ is uncountable. Therefore, for all $y_i\neq y_j\in B_{d(\epsilon_{m_0})}(x_{m_0})$, $d(\phi_1^n(y_i),\phi_1^n(y_j))\leq c,$ for all $n\in \mathbb{N}$ which contradicts the $\aleph_0$-expansiveness of $(X,\phi_{1,\infty})$. Thus, the set $A(X,\phi_{1,\infty})$ is countable.
\end{itemize}

\end{proof}

In the next example, we show that if $X$ is non-compact, then the set of periodic points for an n-expansive or $\aleph_0$-expansive need not be countable.

\begin{exm} \end{exm} 
Let $X=\mathbb{R}$ and $g, h$ on $X$ be defined by: 
\[ g(x) =
2x, \hspace{2mm} \text{for} \ x \in \mathbb{R}  
\]
\[ h(x) =
\frac{1}{2} x, \hspace{2mm} \text{for} \ x \in \mathbb{R}  
\]

Let $(X,\phi_{1,\infty})$ be non-autonomous system such that $\phi_{1,\infty}= \{g,h,g^2,h^2,g^3,h^3,\ldots \ldots\}$ for all $n> 0$. Clearly, $(X,\phi_{1,\infty})$ is both n-expansive and $\aleph_0$-expansive. Now, for every $x\in X$, $\phi_1^{2k}(x)=x$ for all $k\geq 1$. Therefore, every point in $X$ is a periodic for $(X,\phi_{1,\infty})$ and hence the $Per(X,\phi_{1,\infty})=\mathbb{R}$ which is uncountable.

In the following definition we extend the notion of continuum-wise expansiveness \cite{MR1222517} for non-autonomous systems. 

\begin{defn}
The non-autonomous system $(X,\phi_{1,\infty})$ or $\phi_{1,\infty}$ is said to be $\it{continuum wise(cw)}$-$\it{expansive}$ if there exists a $c>0$ such that every non-degenerate subcontinuum $C$ of $X$ satisfies $diam(f_1^n(C))>c$ for some $n\in \mathbb{N}$.
\end{defn}

Similarly, a sequence of homeomorphisms $\phi_n:X\rightarrow X$ is said to be $\it{continuum wise(cw)}$-$\it{expansive}$ if there exists a $c>0$ such that every non-degenerate subcontinuum $C$ of $X$ satisfies $diam(f_1^n(C))>c$ for some $n\in \mathbb{Z}$.

\begin{thm}
Let $(X,d)$ be a metric space and $\phi_n:X\rightarrow X, n=1,2,\ldots$ be a sequence of continuous maps. If $(X,\phi_{1,\infty})$ is $\aleph_0$-expansive, then it is cw-expansive.
\end{thm}

\begin{proof}
Let $\phi_{1,\infty}$ be $\aleph_0$-expansive with constant of $\aleph_0$ expansiveness $c>0$ and $C$ be any non-degenerated subcontinuum of $X$. Thus, $C$ contains uncountable elements. Fix some $x\in X$ and let $S=\{x'\in X:d(\phi_1^i(x),\phi_1^i(x'))\leq c;i>0\}$. Since $f_{1,\infty}$ is $\aleph_0$-expansive, therefore $S$ has at most countable elements and hence there exists some $y\in C\setminus S$ which implies $d(\phi_1^m(x),\phi_1^m(y))> c$ for some $m>0$, and thus diam$(\phi_1^m(C))> c$. Thus, $\phi_{1,\infty}$ is cw-expansive with constant of cw-expansiveness $c>0$. 
\end{proof}

\begin{thm}
Let $(X,d)$ be an uncountable Lindel{\"o}f metric space and $\phi_n:X\rightarrow X, n=1,2,\ldots$ be a sequence of equicontinuous maps. Then $(X,\phi_{1,\infty})$ is never $\aleph_0$-expansive.
\end{thm}

\begin{proof}
Let if possible, $(X,\phi_{1,\infty})$ be $\aleph_0$- expansive with constant of $\aleph_0$-expansiveness $c>0$. Since $\phi_{1,\infty}$ is an equicontinuous system, we can get a $c'>0$ corresponding to $c$ such that $d(\phi_1^i(x),\phi_1^i(y))\leq c$ whenever $d(x,y)\leq c'$ for all $i\in \mathbb{N}$, which implies $B_d(x,c')\subseteq  S_c(x)$, where $S_c(x)=\{y\in X: d(\phi_1^i(x),\phi_1^i(y))\leq c$ for all $i\in \mathbb{Z}\}$. By $\aleph_0$-expansiveness of $(X,\phi_{1,\infty})$, $B_d(x,c')$ is at most countable for every $x\in X$ and  $X$ being Lindel{\"o}f, the open cover $\{B_d(x,c'), x\in X\}$ has a countable sub cover $\{B_d(x_i,c'),i\in \mathbb{N}\}$. Therefore, $X=\bigcup_{i=1}^{\infty}B_d(x_i,c')$ and hence countable, which is a contradiction. Thus, $(X,\phi_{1,\infty})$ is never $\aleph_0$-expansive.
\end{proof}

\begin{thm}
Let $(X,d)$ be a metric space and $\phi_n:X\rightarrow X, n=1,2,\ldots$ be a sequence of continuous maps. Then $\phi_{1,\infty}$ is n-expansive if and only if it is n-expansive on $F\subseteq  X$, where $X\setminus F$ is finite, i.e, $(X,\phi_{1,\infty})$ is n-expansive if and only if $(F,\phi_{1,\infty}|_F)$ is n-expansive.
\end{thm}

\begin{proof}
One can easily observe that if $\phi_{1,\infty}$ is n-expansive on $X$, then it is n-expansive on $F$ also.
Conversely, suppose $(F,\phi_{1,\infty}|_F)$ is n-expansive, then $X\setminus F$ being finite, it suffices to prove the result for $X\setminus F=\{x\}$, i.e. $X\setminus F$ being singleton.
Let $(F,\phi_{1,\infty})$ be n-expansive with constant of n-expansiveness $c>0$. Note that there can exist at most n elements, say $p_1, p_2, p_3,\ldots, p_n$ such that $d(\phi_1^i(x),\phi_1^i(p_j))\leq \delta /2, j=1,2,\ldots,n$, for all $i>0$. For if there exists some $p\in F$ such that $p\neq p_j, j=1,2,\ldots ,n$ and $d(\phi_1^i(x),\phi_1^i(p))\leq c/2$, for all $i>0$, then it contradicts the n-expansiveness of $(F,\phi_{1,\infty}|_F)$. In case no such $p_j$ exist, then any $c'$ such that $0<c'<c/2$ will serve as an n-expansiveness constant for $\phi_{1,\infty}$ on $X$.
\\Suppose there are n such elements, $p_1,p_2,\ldots, p_n$, then choose $0<c'<$min$\{d(p_j,x); j=1,2,\ldots,n\}\leq c/2$ as the n-expansive constant for $\phi_{1,\infty}$ on $X$.
\end{proof}

In the next example we show that the above result need not hold true if $X\setminus F$ is infinite.
\begin{exm} \end{exm} 
Let $X=\mathbb{R}$, $F=\mathbb{Z}$ and $\phi_n, n\in \mathbb{N}$ on $X$ be defined by: 
\[ \phi_n =
x+n, \hspace{2mm} \text{for} \ x \in \mathbb{R}  
\]

Then, the non-autonomous system $(X,\phi_{1,\infty})$ is not n-expansive for any $n\in \mathbb{N}$. However, the system $(F,\phi_{1,\infty}|_{F})$ is n-expansive with any $0<\delta<1$ as an n-expansiveness constant.

For autonomous dynamical systems, it has been proved that:
\\There exists no $\aleph_0$-expansive homeomorphism on a compact interval \cite{MR3694815}.
\\Thus, a natural question arises that whether a similar result holds for the non-autonomous systems. We are yet to answer this question completely, but in some cases we have been able to establish a similar result for non-autonomous systems.

\begin{thm}
Let $X$ be a compact interval of the real line $\mathbb{R}$ and $\phi_n:X\rightarrow X$, $n=1,2,\ldots$ be strictly increasing homeomorphisms, for every $n$  or strictly decreasing homeomorphisms, for every $n$. Then, $(X,\phi_{1,\infty})$ is never $\aleph_0$-expansive.
\end{thm}

\begin{proof}
Let if possible $(X,\phi_{1,\infty})$ be $\aleph_0$-expansive for some non-autonomous system $\phi_{1,\infty}$. Since Fix$(X,\phi_{1,\infty})$ is a closed subset of $X$, therefore $U=X\setminus Fix(X,\phi_{1,\infty})$ is open and hence is a countable union of pair-wise disjoint open intervals $\{I_j\}_{j>0}$. As each $\phi_n$ is either strictly increasing or strictly decreasing, $\phi_{1,\infty}$ is so. Thus, for any $x\in U$ there exists a $j>0$ such that $x\in I_j$ and has converging semi-orbit under $\phi_{1,\infty}$. Therefore, the set $A(X,\phi_{1,\infty})$ is uncountable which is a contradiction by Theorem $3.9[iii.]$ and hence we have $(X,\phi_{1,\infty})$ is never $\aleph_0$-expansive.
\end{proof}

Since by Theorem 3.5, we have $\phi_{1,\infty}$ is $\aleph_0$-expansive if and only if $(\phi_{1,\infty})^k$ is $\aleph_0$-expansive. We get the following result.

\begin{Cor}
Let $X$ be a compact interval of the real line $\mathbb{R}$ and $\phi_n:X\rightarrow X$ be such that $\phi_{2n-1}:X\rightarrow X$, $n=1,2,\ldots$ are strictly increasing and $\phi_{2n}:X\rightarrow X$, $n=1,2,\ldots$ are strictly decreasing homeomorphisms on $X$. Then, $(X,\phi_{1,\infty})$ is never $\aleph_0$-expansive.
\end{Cor}

\section{Meagre-expansiveness for Non-autonomous Systems}
In this section, we define and study give meagre-expansive non-autonomous dynamical systems. We give an important class of meagre-expansive non-autonomous dynamical systems by proving that on a locally connected compact metric space without isolated points, all $\aleph_0$-expansive systems are meagre-expansive.

\begin{defn}
Let $(X,d)$ be a metric space and $\phi_n:X\rightarrow X$ be a sequence of continuous maps for $n=1,2,\ldots$. The non-autonomous system $(X,\phi_{1,\infty})$ or $\phi_{1,\infty}$ is said to be $\it{meagre-expansive}$ if there exists a constant $c>0$ such that for every $x\in X$, the set $\{y\in X:d(\phi_1^i(x),\phi_1^i(y))\leq c,$ for all $i\in \mathbb{N}\}$ is nowhere dense; $c>0$ is called a meagre-expansiveness constant.
\end{defn}
Similarly, a sequence of homeomorphisms $\phi_n:X\rightarrow X$ is meagre-expansive, if for every $x\in X$, the set $\{y\in X:d(\phi_1^i(x),\phi_1^i(y))\leq c,$ for all $ i\in \mathbb{Z}\}$ is nowhere dense.

\begin{defn}
A Borel probability measure $\mu$ of a metric $(X,d)$ is $\it{meagre-expansive}$ with respect to the non-autonomous system $(X,\phi_{1,\infty})$ if there exists $c>0$ such that $\mu(IntS_c(x))$=0, for all $x\in X$, where $S_c(x)=\{y\in X: d(\phi_1^i(x),\phi_1^i(y))\leq c,i\in \mathbb{Z}$\}; $c>0$ is called meagre-expansiveness constant.
\end{defn}

\begin{thm}
Let $(X_1,d_1,\phi_{1,\infty})$ and $(X_2,d_2,\psi_{1,\infty})$ be two non-autonomous systems such that $\phi_{1,\infty}$ is uniformly conjugate to $\psi_{1,\infty}$. If $\phi_{1,\infty}$ is meagre-expansive, then so is $\psi_{1,\infty}$.
\end{thm}
\begin{proof}
Let $c>0$ be a meagre-expansive constant for $\phi_{1,\infty}$. Since $\phi_{1,\infty}$ is uniformly conjugate to $\psi_{1,\infty}$, so there exists a uniform homeomorphism $h:X_1\rightarrow X_2$ such that $h\circ \phi_n=\psi_n\circ h$, for all $n>0$. Thus, $\phi_1^n\circ h^{-1}=h^{-1}\circ \psi_1^n$, for all $n>0$.
As $h^{-1}$ is uniformly continuous, therefore for every $c>0$, there is $c' >0$ such that for $x,y\in X_2$, $d_2(x,y)\leq c' \implies d_1(h^{-1}(x),h^{-1}(y))\leq c$.
\\For a fix $x\in X_2$, the set $S = \{y\in X_2:d_2(\psi^n(x),\psi_1^n(y))\leq c', \forall n>0\} =\{y\in X_2:d_1(h^{-1}(\psi_1^n(x)),h^{-1}(\psi_1^n(y)))\leq c', \forall n>0\} =\{y\in Y:d_1(\phi^n(h^{-1}(x)),\phi_1^n(h^{-1}(y)))\leq c, \forall n>0\}.$ 
\\Since $\phi_{1,\infty}$ is meagre-expansive with meagre-expansiveness constant $c>0$, we can say that $S$ is nowhere dense and hence $\psi_{1,\infty}$ is meagre-expansive with meagre-expansiveness constant $c'>0$.
\end{proof}

\begin{Cor}
Let $(X,d)$ be a compact metric space. If $(X,\phi_{1,\infty})$ is meagre-expansive, then so is $(X,\psi_{1,\infty})$, where $\psi_{1,\infty}=\{h\circ \phi_n\circ h^{-1}\}_{n=1}^{\infty}$ and $h:X\rightarrow X$ is a homeomorphism. 
\end{Cor}

\begin{thm}
Let $(X,d)$ be a metric space and $\phi_n:X\rightarrow X, n=1,2,\ldots$ be a sequence of equicontinuous maps. For any positive integer $k$, the non-autonomous system $(X,\phi_{1,\infty})$ is meagre-expansive if and only if $(\phi_{1,\infty})^k$ is meagre-expansive.
\end{thm}

\begin{proof}
Let $\delta>0$ be a meagre-expansive constant for $f_{1,\infty}$. Since $\{\phi_n\}_{n=1}^\infty$ is an equicontinuous family of maps, for any $m\geq 0$, $mk+1\leq j\leq (m+1)k$; $\phi_j$ is uniformly continuous on $X$ and thus there exists $\delta_j>0$ such that \begin{align*}
d(x,y)\leq\delta_j\implies d(\phi_{mk+1}^j(x),\phi_{mk+1}^j(y))\leq\delta.
\end{align*} 
Since $\phi_n:X\rightarrow X,n=1,2,\ldots$ are equicontinuous maps, $c_j$ doesn't depend on m. Take $c'=min\{c_j:mk+1\leq j \leq (m+1)k\}$. So for any $m\geq 0$, $d(x,y)\leq c'\implies d(\phi_{mk+1}^j(x),\phi_{mk+1}^j(y))\leq c.$ Let $(\phi_{1,\infty})^k=\psi_{1,\infty}=\{\psi_n\}_{n=1}^{\infty}$, where $\psi_n=\phi_{(n-1)k+1}^{nk}$ and $\psi_1^n=\psi_n\circ \ldots \circ \psi_1$.
Note that $\phi_1^{nk}=\psi_1^n$. Thus for any $m\geq 0$ and $mk\leq j\leq (m+1)k$, $d(\psi_1^m(x),\psi_1^m(y))\leq c$ which implies $d(\phi_1^{mk}(x),\phi_1^{mk}(y))\leq c'$ and hence we get that $d(\phi_{mk+1}^j(\phi_1^{mk}(x)),\phi_{mk+1}^j(\phi_1^{mk}(y)))\leq c$ which implies $d(\phi_1^j(x),\phi_1^j(y))\leq c$. Since $c$ is meagre-expansive constant for $f_{1,\infty}$, the set $\{y\in X:d(\phi_1^i(x),\phi_1^i(y))\leq c ;i>0\}$ is nowhere dense. Therefore, $\{y\in X:d(\psi_1^i(x),\psi_1^i(y))\leq c' ;i>0\}$ is nowhere dense and hence $(\phi_{1,\infty})^k$ is meagre-expansive with constant of meagre-expansiveness $c'>0$.
\\Conversely, if $(\phi_{1,\infty})^k$ is meagre-expansive with constant of meagre-expansiveness $c>0$, then for any $x\in X$, the set $\{d(\psi_1^n(x),\psi_1^n(y))\leq c,n>0\}$ is nowhere dense, where $\psi_n=\phi_{(n-1)k+1}^{nk}$. Thus, the set $\{y\in X:d(\phi_1^{nk}(x),\phi_1^{nk}(y))\leq c,$ for all $n>0\}$ = $\{d(\phi_1^j(x),\phi_1^j(y))\leq c,j> 0\}$ is nowhere dense. Hence, $(\phi_{1,\infty})$ is meagre-expansive with constant of meagre-expansiveness $c>0$. 
\end{proof}

\begin{thm}
Let $(X,\phi_{1,\infty})$ be a non-autonomous system, where $\phi_{1,\infty}$ is a family of self-homeomorphisms. Then $(\phi_{1,\infty})$ is meagre-expansive if and only if $(\phi_{1,\infty})^{-1}$ is meagre-expansive.
\end{thm}

\begin{proof}
Let $c>0$ be a meagre-expansive constant for $\phi_{1,\infty}$. For a fixed $x\in X$; the set \begin{align*}
\{y\in X: d(\phi_1^i(x),\phi_1^i(y))\leq c,i\in \mathbb{Z}\}
\end{align*} is nowhere dense, i.e, the set
\begin{align*}
\{y\in X: d((\phi_1^i)^{-1}(x),(\phi_1^i)^{-1}(y))\leq c,i\in \mathbb{Z}\}
\end{align*} is nowhere dense. Thus, $(\phi_{1,\infty})^{-1}$ is meagre-expansive.
\\The converse can proved similarly.
\end{proof}

\begin{Cor}
The non-autonomous system $(X,\phi_{1,\infty})$ on a compact metric space $X$ is meagre-expansive if and only if $(\phi_{1,\infty})^{k}$ is meagre-expansive, for $k\in \mathbb{Z}-\{0\}$, where $\{\phi_n\}_{n=1}^{\infty}$ is a family of equicontinuous self-homeomorphisms on $X$. .
\end{Cor}

\begin{thm}
Let $(X,d)$ be a locally connected metric space without isolated points and $\phi_n:X\rightarrow X, n=1,2,\ldots$ be a sequence of continuous maps. If $(X,\phi_{1,\infty})$ is cw-expansive, then it is meagre-expansive.
\end{thm}

\begin{proof}
Let $\phi_{1,\infty}$ be cw-expansive with constant of cw-expansiveness $c>0$. Choose $0<c'
<c/2$. We claim that $c'$ will be a meagre expansiveness constant for $\phi_{1,\infty}$. If our claim fails to hold, then there exists some $x\in X$ such that the set $S=\{y\in X:d(\phi_1^i(x),\phi_1^i(y))\}\leq c',i> 0\}$ is not nowhere dense, i.e., $Int(\overline{S})\neq \emptyset$. As $X$ is locally connected, we can find an open connected subset $U$ of $X$ such that $x\in U\subset$ Int$(\overline{S})$, which implies $U\subseteq \overline{S}$ and hence $d(\phi_1^i(x_1),\phi_1^i(x_2))\leq d(\phi_1^i(x_1),\phi_1^i(x))+ d(\phi_1^i(x),\phi_1^i(x_2)\leq 2\c' \leq c$, for all $i>0$, where $x_1, x_2\in U$. Therefore, we get that $diam(\phi_1^i(U))<c$, for some $i> 0$. Since $c$ is a cw-expansiveness constant for $\phi_{1,\infty}$, we get that $U=\{x\}$ which further implies $x$ is an isolated point of $X$ which contradicts our hypothesis. Thus, $\phi_{1,\infty}$ is meagre-expansive with a constant of meagre-expansiveness $c'$.  
\end{proof}

The above result together with Theorem 3.10 gives us the following Corollary.

\begin{Cor}
Let $(X,d)$ be a locally connected metric space without isolated points and $\phi_n:X\rightarrow X, n=1,2,\ldots$ be a sequence of continuous maps. If $(X,\phi_{1,\infty})$ is $\aleph_0$-expansive, then it is meagre-expansive.
\end{Cor}

In the next example we show that the converse of the above Corollary need not be true. 

\begin{exm} \end{exm}
Consider $X$ to be a three torus $T^3= S^1 \times T^2 $, where $S^1$ is the unit circle and $T^2$ is the two torus.
Define $g_1: T^3\rightarrow T^3$ by $g_1(\theta,z)=(\theta,g(z))$, for some $(\theta, (z))\in S^1 \times T^2$, where $g:T^2\rightarrow T^2$ is defined as 
\\$g(z)=g(z_2,z_2)= \begin{bmatrix}
2 & 1 \\ 1& 1
\end{bmatrix} \begin{bmatrix}
z_1 \\ z_2 
\end{bmatrix}$ mod 1. Let $g_2$ be the identity map of the three torus. \\We define the non-autonomous system $(X,\phi_{1,\infty})$ as follows:
\[ \phi_n(x) = \begin{cases}
g_1,  & \text{for} \ n \hspace{0.1cm} odd \\
g_2, & \text{for} \ n\hspace{0.1cm} even.
\end{cases} \]
Since $g$ is an expansive map of the two torus, therefore $S_{\delta}(\theta, z)=[\theta-\delta, \theta+\delta]\times \{z\}$, for all $(\theta, z)\in T^3$ and for all $\delta>0$, where $S_{\delta}(\theta, z)=\{(\theta',y):d(f_1^i(\theta,z),f_1^i(\theta',y))<\delta, i \in \mathbb{Z} \}$.
Thus, we have that $(X,\phi_{1,\infty})$ is meagre expansive but not $\aleph_0$-expansive taking  expansiveness constant $\delta>0$ as above in both cases. 
\begin{rmk}The following diagrams shows the relations among various forms of expansiveness in a non-autonomous discrete system defined on a locally connected metric space without isolated points:
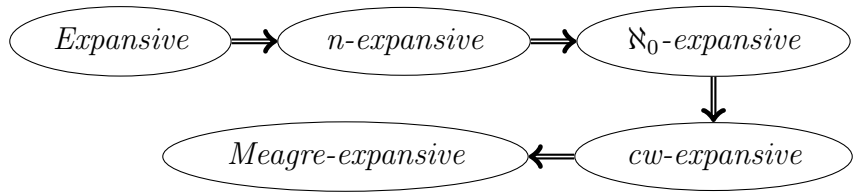
\tikzset{
  startstop/.style={
    rectangle, 
    rounded corners,
    minimum width=3cm, 
    minimum height=1cm,
    align=center, 
    draw=black, 
    fill=red!30
    },
  process/.style={
    rectangle, 
    minimum width=3cm, 
    minimum height=1cm, 
    align=center, 
    draw=black, 
    fill=blue!30
    },
  decision/.style={
    rectangle, 
    minimum width=3cm, 
    minimum height=1cm, align=center, 
    draw=black },
  arrow/.style={thick,->,>=stealth},
  dec/.style={
    ellipse, 
    align=center, 
    draw=black},
}
\tikzstyle{line} = [draw, ->, double, thick]

\begin{center}

%\resizebox{!}{.8\textheight}{% if required
\begin{tikzpicture}[
  start chain=going below,
  every join/.style={arrow},
  node distance=0.6cm
  ]
\node (in1) [dec,on chain] {Expansive};
\node (out1) [dec, right=of in1, node distance=5cm] {n-expansive};
\node (out2) [dec, right=of out1,node distance=4cm] {$\aleph_0$-expansive};
\node (out3) [dec,below=of out2,node distance=4cm] {cw-expansive};
\node (out4) [dec,left=of out3,node distance=4cm] {Meagre-expansive};

\path [line] (in1) -- (out1);
\path [line] (out1) -- (out2);
\path [line] (out2) -- (out3);
\path [line] (out3) -- (out4);

\end{tikzpicture}
\end{center}
The converse implication does not hold in any of the cases. We have given Example 4.1 which shows meagre-expansiveness need not imply $\aleph_0$-expansiveness. 
\\There are examples of cw-expansive autonomous systems which are not expansive \cite{MR1222517} and n-expansive autonomous systems which are not expansive\cite{MR3694815}. It will be interesting to construct such examples for non-autonomous discrete systems.

\end{rmk}

\begin{thm}
Let $(X,d)$ be a Lindel{\"o}f metric space and $\phi_n:X\rightarrow X, n=1,2,\ldots$ be equicontinuous homeomorphims. Then, $(X,\phi_{1,\infty})$ is never meagre-expansive.
\end{thm}

\begin{proof}
We first claim that $\mu(Stab(X,\phi_{1,\infty}))=0$, for every meagre-expansive homeomorphism $\mu$ of $\phi_{1,\infty}$. Let $\mu$ be a meagre-expansive measure for $\phi_{1,\infty}$ with meagre-expansiveness constant $c>0$. Let $x\in Stab(X,\phi_{1,\infty})$, then $x\in IntS_c(x)$, where $S_c(x)=\{y\in X: d(\phi_1^i(x),\phi_1^i(y))\leq c,i\in \mathbb{Z}\}$. Thus, we get $\{Int(S_c(x)):x\in Stab(X,\phi_{1,\infty})\}$ is an open cover of $Stab(X,\phi_{1,\infty})$ and since $X$ is Lindel{\"o}f, there exist $x_1, x_2,\ldots \in$ Stab$(X,\phi_{1,\infty})$ such that Stab$(X,\phi_{1,\infty})\subseteq  \bigcup_{i=1}^{\infty}$Int$S_c(x_i)$ and hence $\mu$(Stab$(X,\phi_{1,\infty}))=0$.
\\Now, assume that $(X,\phi_{1,\infty})$ is meagre-expansive with constant of meagre-expansiveness $c>0$ and $\mu$ be any Borel measure of $X$. Then, Int$S_c(x)$ being empty, $\mu($Int$S_c(x))=0$, for every $x\in X$. Thus, $\mu$ is a meagre-expansive measure for $\phi_{1,\infty}$. As $\phi_n, n=1,2,\ldots$, is equicontinuous, we have $Stab(X,\phi_{1,\infty})=X$ and by the above claim, $\mu(X)=0$, which is a contradiction. Thus, $(X,\phi_{1,\infty})$ is never meagre-expansive.
\end{proof}

It is easy to observe that for the non-autonomous system $(X,\phi_{1,\infty})$, Fix$(X,\phi_{1,\infty})\subseteq  $ Stab$(X,\phi_{1,\infty})$ and hence we have the following remark:

\begin{rmk}
For a non-autonomous system of homeomorphisms on a Lindel{\"o}f metric space, $\mu$(Int(Fix$(X,\phi_{1,\infty})))=0$, with respect to every meagre expansive measure $\mu$ of $\phi_{1,\infty}$.  
\end{rmk}

\section*{Acknowledgement}
The first author is funded by GOVERNMENT OF INDIA, MINISTRY OF SCIENCE and TECHNOLOGY No: DST/INSPIRE Fellowship/[IF160750].

\end{document}